\documentclass[11pt,a4paper]{article}

\usepackage{amsmath,amsthm,amsopn,amsfonts,amssymb,dsfont,color}
\usepackage{psfrag}
\usepackage{geometry}
\usepackage{graphicx}
\usepackage[utf8]{inputenc}
\usepackage[english]{babel}
\usepackage[active]{srcltx}
\usepackage{subcaption}

\usepackage{footmisc}

\usepackage{a4,a4wide}
\usepackage{color}
\usepackage{enumerate}

\def\spt{{\rm spt}}
\def\R{\mathbb R}

\newtheorem{theo}{\textbf{Theorem}}[section]

\newtheorem{lem}{\textbf{Lemma}}[section]
\newtheorem{prop}[theo]{\textbf{Proposition}}

\newtheorem{defi}{\textbf{Definition}}[section]

\numberwithin{equation}{section}

\title{Terrace solutions for non-Lipschitz multistable nonlinearities}
\date{\today}
\author{Thomas Giletti \footnote{Institut Elie Cartan de Lorraine, UMR 7502, University of Lorraine, 54506 Vandoeuvre-les-Nancy, France
thomas.giletti@univ-lorraine.fr}, Ho-Youn Kim\footnote{Department of Mathematical Sciences, KAIST, 291 Daehak-ro, Yuseong-gu, Daejeon, 34141,
Korea.}, Yong-Jung Kim$^\dagger$ }

\begin{document}

\maketitle

\begin{abstract}
Traveling wave solutions of reaction-diffusion equations are well-studied for Lipschitz continuous monostable and bistable reaction functions. These special solutions play a key role in mathematical biology and in particular in the study of ecological invasions. However, if there are more than two stable steady states, the invasion phenomenon may become more intricate and involve intermediate steps, which leads one to consider not a single but a collection of traveling waves with ordered speeds. In this paper we show that, if the reaction function is discontinuous at the stable steady states, then such a collection of traveling waves exists and even provides a special solution which we call a terrace solution. More precisely, we will address both the existence and uniqueness of the terrace solution.
\end{abstract}

\section{Introduction}

The solution of an initial value problem of ordinary differential equations (or ODEs for brevity),
\[
 U ' =F(t,U),\quad U(t_0)=U_0,\quad t\in\R,
\]
uniquely exists at least locally in time if the reaction function $F(t,U)$ is Lipschitz continuous in a neighborhood of $(t_0,U_0)$. This fundamental ODE theory is the foundations to obtain the uniqueness and the existence for many problems of partial differential equations (or PDEs), which is the reason why the Lipschitz continuity is mostly assumed. However, the Lipschitz continuity also gives some undesirable phenomena. For example, when a solution converges to a steady state, it does so only asymptotically and never arrives at it in a finite time. 

The purpose of this paper is to develop basic theories for the interaction of traveling wave solutions of a reaction-diffusion equation,
\begin{equation}\label{eq:rd}
\partial_t u = \partial_{xx} u + f(u),\quad  t >0, \ \  x\in\R,
\end{equation}
when the reaction function $f$ has several stable steady states and is discontinuous at them. Here by stable steady state we typically mean that it is asymptotically stable with respect to the ODE $u' = f(u)$. In ecology and population dynamics, the unknown function $u$ typically stands for a species density. Due to the several steady states, propagation phenomena may involve intermediate steps and some consecutive traveling wave solutions (see \cite{DGM,GM,Polacik} in the Lipschitz continuous case).

Discontinuities may come from harvesting terms and make the model more realistic, for instance by providing finite time extinction (see \cite{Kim2020}). Since $f$ is not Lipschitz continuous, the classical theory for the existence and the uniqueness fail. However, the non-Lipschitz reaction function $f$ does not cause too much trouble since it has discontinuities at stable steady states only. On the other hand, the discontinuities in $f$ produce a sharp interface of traveling wave solutions and make it possible to glue together traveling wave solutions. Furthermore, the asymptotic drift of a logarithmic scale between two consecutive waves, that one observes in the Lipschitz case, disappears.

We take three hypotheses for the reaction function $f$. First, we assume that there exists a finite number of steady states, $\theta_i$'s for $i=0,1,\cdots,2I$, such that\
\begin{eqnarray}
 && 1 = \theta_0 > \theta_1 > \cdots > \theta_{2I-2} > \theta_{2I-1} > \theta_{2I} = 0, \nonumber    \\
 &&  f(u)>0\ \text{ for }\ u<\theta_{2I}\ \ \text{ and }\ \theta_{2i+1}<u<\theta_{2i},\quad i=0,\cdots,I-1,  \label{H1}\\
&& f(u)<0\ \text{ for }\ u>\theta_{0}\ \text{ and }\ \theta_{2i}<u<\theta_{2i-1},\quad i=1,\cdots,I.\nonumber
\end{eqnarray}
A diagram of such $f(u)$ is described in Figure~\ref{fig1}~(a). We have chosen 0 and 1 as the extremal steady states for our convenience, which is always possible by rescaling the solution. Under the assumptions in \eqref{H1}, the $\theta_{2i}$'s are stable steady states for $i=0,\cdots,I$, and the $\theta_{2i-1}$'s are unstable ones for $i=1,\cdots,I$.

A solution $u(t,x)$ of~\eqref{eq:rd} is called a traveling wave solution connecting $1$ to $0$ if there exists a wave profile $\phi$ and a wave speed $c\in\R$ such that
\[
u(t,x)=\phi(x-ct),\   \text{ with } \  \phi(z)\to1\ \text{ as }\ z\to-\infty,\ \text{ and }\ \phi(z)\to0\ \text{ as }\ z\to\infty.
\]
Traveling wave solutions have been intensively studied when $f$ is Lipschitz continuous. In particular, if $I=1$, the nonlinearity is called bistable and there exists a unique traveling wave speed $c \in \R$, and a unique (up to translation) wave profile $\phi$. Furthermore, it is well-known that these traveling wave solutions describe the large-time dynamics of solutions of the Cauchy problem for large classes of initial data. In particular they are useful to understand a large range of propagation phenomena from physics, biology and population dynamics, which can be modeled by reaction-diffusion equations such as~\eqref{eq:rd}. We refer to the celebrated works~\cite{AW75,FifeMcLeod77} for more details.

If $I >1$, a traveling wave solution connecting $1$ to $0$ does not exist in general, and a so-called propagating terrace is considered instead. This notion refers to a collection of traveling waves that connect steady states sequentially from $1$ to $0$; we refer again to~\cite{FifeMcLeod77} where it is introduced under the different name of ``minimal decomposition'', to~\cite{Polacik} for further developments in the homogeneous case and to~\cite{DGM,GM} where propagating terraces have been studied in the context of spatially heterogeneous reaction-diffusion equations. An adaptation to a discontinuous case will be given below in Definition~\ref{terracedef}. Note that, while the propagating terrace still dictates the large-time behavior of solutions of the Cauchy problem, it does so only locally since it is not a single but a family of solutions of~\eqref{eq:rd}. Concerning this latter fact, in the discontinuous reaction framework we will obtain a single function that connects the external steady states $1$ and $0$ and provides the global picture of solution dynamics.

Next, we assume that $f$ has the regularity of
\begin{equation}\label{H2}
f\in C^1(\R\setminus\{\theta_{2i}\,| \ 0\leq i \leq I \})\cap {\rm Lip}\, (\R\setminus\{\theta_{2i}\,| \ 0\leq i \leq I \}).
\end{equation}
Notice that, unlike in the aforementioned works, we do not assume any regularity at the stable steady states. In fact, we assume $f$ has jumping discontinuities at the stable steady states such as
\begin{equation}\label{H3}
\lim_{u \to \theta_{2i}^-}f(u)>0>\lim_{u \to \theta_{2i}^+}f(u),\quad i=0,\cdots,I.
\end{equation}
As a matter of fact, it is such discontinuities that allow a ``terrace solution" of \eqref{eq:rd} when $f$ satisfies \eqref{H1}--\eqref{H3}.

We made some of the above assumptions to simplify the presentation. For example, instead of~\eqref{H2}, we can allow $f$ to have a finite number of discontinuities away from steady states without any incidence on the results. However, if $f$ is discontinuous at an unstable steady state $\theta_{2i-1}$, the solution of~\eqref{eq:rd} is not unique and one would face well-posedness issues. Moreover, in~\eqref{H3}, we assume that the left and the right side limits of $f$ at the stable steady states $\theta_{2i}$ are both nonzero. This is the key assumption of this paper, though one might actually extend our arguments to reaction functions which are only H\"older continuous at stable steady states but not Lipschitz continuous. If a Lipschitz continuous reaction function is given, one might consider the discontinuous nonlinearity in this paper as an approximation (see Figure~\ref{fig1}); yet as we mentioned above, a discontinuous reaction function exhibits more realistic phenomena in some aspects such as finite time extinction in population dynamics.\\

\begin{figure}[h]
	\begin{subfigure}{.62\textwidth}
		\centering
		\includegraphics[width=.8\linewidth]{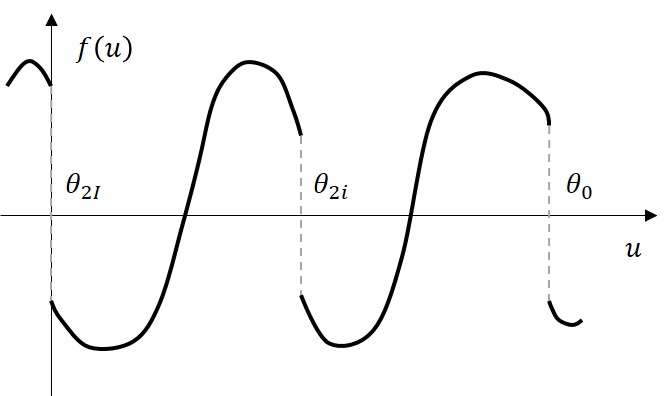}
		\label{fig1:sfig1}
	\end{subfigure}%
	\begin{subfigure}{.38\textwidth}
		\centering
		\includegraphics[width=.8\linewidth]{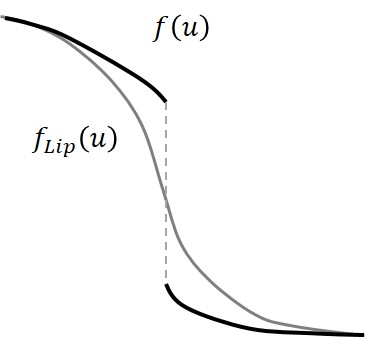}
		\label{fig1:sfig2}
	\end{subfigure}
	\caption{(a) An example of a non-Lipschitz multistable reaction function with $I=2$. (b) Any Lipschitz continuous reaction function can be approximated by discontinuous reaction functions. }
	\label{fig1}
\end{figure}

Under the three assumptions \eqref{H1}--\eqref{H3}, we will prove the existence and uniqueness (up to some shifts) of a terrace solution. Roughly speaking, a terrace solution is a special solution in which the whole profile is separated by plateaus into several sub-profiles and each sub-profile moves at its own speed; see Definitions~\ref{terracedef} and~\ref{terracesolution}. By analogy with the case of a smooth reaction~\cite{DGM,FifeMcLeod77}, we also expect that these terrace solutions appear in the long-time asymptotics of solutions of the Cauchy problem; this will be the subject of a future work.

\section{Definitions and main results}

In this section, we introduce key definitions, basic properties together, and our main results. 
In particular, we define some special solutions, namely traveling waves and terraces.
\begin{defi}\label{def:twsol}
	A $C^1$ function $\phi:\R\to\R$ is called a {\bf traveling wave solution} of \eqref{eq:rd} if there exists $c\in\R$ such that
\begin{equation}\label{TWeqn}
\phi''+c\phi'+f(\phi)=0
\end{equation}
in the classical sense in the domain $\{z\in\R:\phi(z)\ne\theta_{2i},\  0 \leq i \leq I\}$. We call the constant $c$ a {\bf traveling wave speed}. Furthermore, we say that a traveling wave solution $\phi$ monotonically connects two steady states $\theta_i$ and $\theta_j$ with integers $0 \leq i < j \leq 2I$ if it is a decreasing function such that
\begin{equation*}\label{limits0}
\lim_{z \to-\infty}\phi(z)=\theta_i , \quad  \lim_{z\to\infty}\phi(z)=\theta_j.
\end{equation*}
\end{defi}
One may check that, if $\phi$ is a traveling wave solution, then $u(t,x)=\phi(x-ct)$ is a weak solution of \eqref{eq:rd}; see \cite{Kim2020} for details for the solution notion with the discontinuous reaction of the paper. We also use the following definition for the solution of \eqref{eq:rd}:
\begin{defi}\label{def:sol}
	A function $u(t,x)$ is called a {\bf solution} of \eqref{eq:rd} if it is $C^0$ with respect to~$t$, $C^1$ with respect to $x$, and satisfies \eqref{eq:rd} in the classical sense in the domain $$\{(t,x)\in\R^+\times\R : \ u(t,x) \ne \theta_{2i}, \ 0 \leq i \leq I.\}.$$
\end{defi}
We call both $\phi(z)$ and $u(t,x)=\phi(x-ct)$ traveling wave solutions.

\begin{defi}\label{connectedcompact} The support of a function $\psi:\R\to\R$ is the set of points where $\psi$ is nonzero, that is
\begin{equation}\label{support}
\text{spt} \, (\psi )= \{ z \in \mathbb{R} \, | \ \psi (z) \neq 0 \}.
\end{equation}
(Note that we do not use closed support.) A traveling wave solution $\phi$ is called: $(i)$ {\bf connected} if the support of $\phi'$ is connected; $(ii)$ {\bf compact} if the closure of the support of $\phi'$ is compact.
\end{defi}

If $f$ satisfies \eqref{H1} and \eqref{H2} only and is Lipschitz continuous, then every monotone and nontrivial traveling wave solution $\phi$ is connected and $\spt(\phi')=\R$. Hence, Definition \ref{connectedcompact} is mostly meant for the case with the discontinuity hypothesis~\eqref{H3}. In the bistable case $I=1$, the existence and uniqueness of a traveling wave solution that monotonically connects 1 and~0 has been addressed by one of the authors in \cite{KimChung}.

\begin{lem}[Chung and Kim \cite{KimChung}] Suppose that $f$ satisfies \eqref{H1}--\eqref{H3} with $I=1$. Then, a traveling wave solution of \eqref{eq:rd} that monotonically connects 1 and 0 exists and is unique up to a translation. Furthermore, this traveling wave solution is connected and compact.
\end{lem}
On the other hand, if there are other stable steady states, such a traveling wave connecting directly 1 and 0 may or may not exist. The notion of a propagating terrace is considered specifically to handle such a situation.
\begin{defi} \label{terracedef} A collection of connected traveling wave solutions $\{\phi_j:j=1,\cdots,J\}$ is called a {\bf propagating terrace} connecting $1$ and $0$ if each $\phi_j$ monotonically connects two steady states $\theta_{i_{j-1}}$ and $\theta_{i_{j}}$, and these limits and the wave speeds $c_j$ corresponding to $\phi_j$ satisfy
\begin{equation*}\label{2.2}
1=\theta_{i_0}>\theta_{i_1}>\theta_{i_2}>\cdots>\theta_{i_J}=0
\quad\text{and} \quad c_1\le\cdots\le c_J.
\end{equation*}
The steady states $\theta_{i_j}$'s are called the {\bf platforms} of the terrace.
\end{defi}
Note that if two traveling waves of the terrace are compact and have the same speed, that is $c_i = c_{i+1}$ for some $i$, one may take the two traveling fronts $\phi_i$ and $\phi_{i+1}$ as a single traveling front and the propagating terrace would not be unique. However, we have imposed that the traveling waves in a terrace are connected with the support defined by \eqref{support}. This implies that we consider all traveling fronts with a same speed as separated traveling fronts in Definition \ref{terracedef}.

As we will prove below, the left and right limits of all traveling waves constituting a propagating terrace connecting $1$ and $0$ must be stable steady states, and due to the discontinuities of $f$ it will follow that these traveling waves are compact in the sense of Definition~\ref{connectedcompact}. In particular, these can be ``glued'' into a single special solution of \eqref{eq:rd}, which leads to the next definition.

\begin{defi} \label{terracesolution}
Let $f$ satisfy \eqref{H1}--\eqref{H3}, $\{\phi_j:j=1,\cdots,J\}$ be a propagating terrace, and $c_j$'s be the corresponding wave speeds. For given translation variables $\xi\in\R^J$, the summation,
\begin{equation*}\label{Phi}
\Phi(t,x; \xi):=\sum_{1 \leq j \leq J} (\phi_j (x - \xi_j - c_j t) - \theta_{i_{j}}),
\end{equation*}
is called a {\bf terrace function}. If moreover $\Phi (\cdot,\cdot; \xi)$ solves \eqref{eq:rd} (possibly only for $t>T$ with $T>0$), then we call it a \textbf{terrace solution}.
\end{defi}
\begin{figure}[h]
	
	\centering
	\includegraphics[width=.6\linewidth]{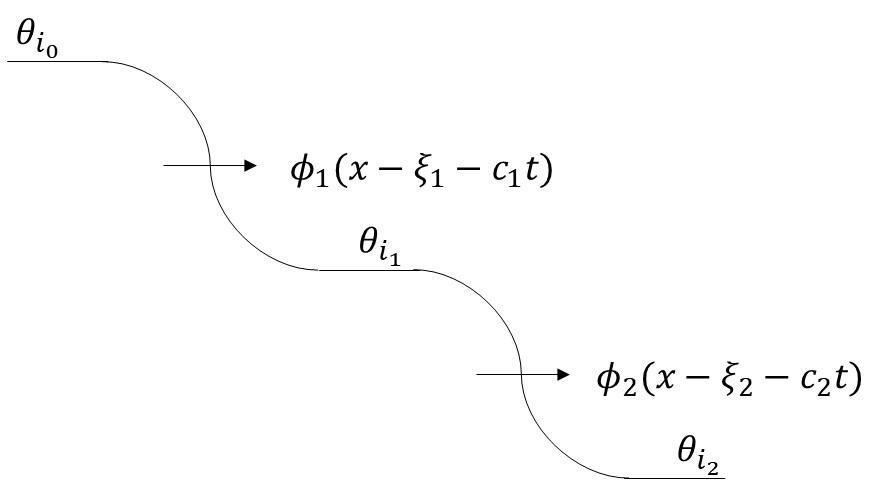}
	\caption{An example of a terrace function with two fronts.}
	\label{fig3}
\end{figure}

We refer to Figure~\ref{fig3} for an illustration of a terrace function involving two compact and connected traveling wave solutions.

If \eqref{H3} fails and $f$ is Lipschitz-continuous, and if $J \geq 2$, then the terrace function is not a solution for any $T>0$ and $\xi\in\R^J$. On the other hand, if all $\phi_j$'s in the terrace are compact, then one can always find shifts $\xi \in \R^J$ such that the terrace function is a solution. Hence, the existence of the terrace solution is the unique property obtained by breaking the Lipschitz continuity. \\

Let us now turn to the statement of the main results of the paper. The main theorem concerns the existence and uniqueness of the propagating terrace.

\begin{theo}[Existence and uniqueness of a terrace]\label{th:exist} Let $f$ satisfy \eqref{H1}--\eqref{H3}. There exists a propagating terrace connecting $1$ and $0$. It is unique in the sense that the set of platforms $\{ \theta_{i_j}\}$ is unique, and for each $j$ the traveling wave solution $\phi_j$ is unique up to translation.

Furthermore, the propagating terrace satisfies the following properties:
\begin{enumerate}[$(i)$]
\item the $\theta_{i_j}$'s are stable steady states, i.e. $i_j$'s are even integers, for all $j=0,\cdots,J$;
\item the traveling waves $\phi_j$ are monotone, connected, and compact for $j=1,\cdots,J$ in the sense of Definition~\ref{connectedcompact}.
\end{enumerate}
In particular, in the bistable case when $I=1$, the propagating terrace consists of a single traveling wave connecting $1$ and $0$.
\end{theo}
It follows from Theorem~\ref{th:exist}, in particular the property $(ii)$, that there exist terrace solutions of~\eqref{eq:rd} in the sense of Definition~\ref{terracesolution}. We conjecture that the solution of the Cauchy problem converges, for a large class of initial data, to such a terrace solution. This will be adressed in a future work~\cite{GK}, and it would be a new result even in the bistable case where the terrace contains a single traveling wave. This would be consistent with the case of a smooth reaction function~\cite{DGM,FifeMcLeod77,GM}, where some piecewise (or local) convergence of the solution toward each of the traveling waves of the propagating terrace was shown. However, the novelty of our approach using a non-Lipschitz reaction function is that the global dynamics of a solution is dictated by a single terrace solution, which is a special solution of \eqref{eq:rd}.

\paragraph{Plan of the paper:} In Section~\ref{sec:exist}, we prove the existence part of Theorem~\ref{th:exist} using a phase plane analysis, which has to be done with extra caution to handle the discontinuities of the reaction function. Moreover, we will use an iterative method to construct all the traveling waves of the propagating terrace, starting from the uppermost one. Then, Section~\ref{sec:unique} is devoted to the uniqueness of the terrace (up to some shifts) and completes the proof of Theorem~\ref{th:exist}.

\section{Existence of terrace}\label{sec:exist}

In this section we address the existence of terrace in the multistable case under assumptions \eqref{H1}--\eqref{H3}. In the usual Lipschitz continuous case, there are several proofs for the existence of traveling waves, drawing on either phase plane analysis~\cite{AW75,FifeMcLeod77}, dynamical systems~\cite{FangZhao,GilettiRossi,Weinberger}, or intersection number arguments~\cite{DGM,KPP}. Due to the spatial homogeneity of \eqref{eq:rd} and in spite of the discontinuities in the reaction function, we adopt the former approach of phase plane analysis.

A traveling wave solution with speed $c$ satisfies
\begin{equation}\label{eq:ode}
p '' + c p' + f(p) =0
\end{equation}
together with some appropriate asymptotics as $x\to\pm\infty$, where $c$ is also an unknown. It is convenient to rewrite the second order equation as a first-order ODE system:
\begin{eqnarray}
\label{phaseplane1} p' &=& q, \\ \label{phaseplane2} q' &=& -cq-f(p).
\end{eqnarray}
This system immediately shows that the discontinuity in $f$ does not break the well-posedness of the problem \eqref{eq:ode} as long as $q\ne0$ when $p=\theta_{2i}$, the discontinuity point of $f$ (the solution can easily be extended by $C^1$-regularity). Yet, as a consequence, solutions of the ODE system \eqref{phaseplane1}--\eqref{phaseplane2} are continuous and piecewise $C^1$.

Beforehand, we introduce
$$M: = \| f \|_\infty<\infty.$$
Let $p_u$ be a positive stable steady state, i.e.
\[
p_u \in \{  \theta_{2i} \ : \ i \in \{ 1,\cdots,I\} \}.
\]
If the reaction function is Lipschitz continuous, one can connect two steady states only asymptotically, but one cannot take a steady state as an initial value to obtain a nontrivial solution. Hence, it is impossible to connect more than two steady states even asymptotically. One of the advantages of taking a non-Lipschitz reaction function is that one can take a steady state as an initial value and use the phase plane analysis technique explicitly as one can see in the followings.

For any $c \in \R$, there is a solution of~\eqref{eq:ode} which is identical to $p_u$ on the left half line and smaller than $p_u$ on an interval in the right half line. It is simply obtained by solving the problem
\begin{equation}\label{eq3.4}
  \tilde{p}'' + c \tilde{p}' + \tilde{f} (\tilde p) =0,\quad
  \tilde{p}(0 ) = p_u, \quad \tilde{p}' (0) = 0,
\end{equation}
where
\[
\tilde{f}(p)=\left\{\begin{array}{cc}
               f(p), & p<p_u, \\
               \lim_{s\to p_u^-} f(s), & p\ge p_u.
             \end{array}
\right.
\]
Notice that $\tilde f$ is Lipschitz continuous in a neighborhood of $p_u$, $\tilde{p}'' (0) < 0$, and hence $\tilde{p}'(\epsilon)<0$ for small $\epsilon>0$. The positive solution of \eqref{eq3.4} may exist as long as $\tilde{p}'<0$ and $p>0$. Then, the maximal point of the solution domain is given by
\[
X_c = \sup\,  \{ x >0 \, | \ \tilde{p} ' (x) < 0 \  \mbox{ and } \ \tilde{p} (x) >0  \} \in (0,\infty].
\]
Since $f(p)=\tilde f(p)$ for all $p<p_u$ and the solution satisfies $p(x)<p_u$ for all $x\in(0,X_c)$,  it is a positive decreasing solution of \eqref{eq:ode}, or equivalently of \eqref{phaseplane1}--\eqref{phaseplane2}, on the interval $(-\infty,X_c)$.

Moreover, since the solution $p(x)$ decreases strictly on $(0,X_c)$, we may consider $q$ as a function of $p \in (p_l,p_u)$ with $p_l := \lim_{x \to X_c} p (x)$ and $q<0$ on $(p_l,p_u)$. Then, we have
$$
{dq \over dp} = -c-{f(p) \over q} \quad \mbox{ on } \ (p_l,p_u),
$$
together with $q (p_u) = 0$, and either $p_l = 0$ or $q (p_l) = 0$. As explained above, the solution is understood as continuous and piecewise $C^1$ function due to the discontinuities of $f$ at the stable steady states.

We sum up the above in the following definition and proposition:
\begin{defi}\label{defi:trajectory} Let $f$ satisfy \eqref{H1}--\eqref{H3}. We say that $q=q(p)$ solves
\begin{equation}\label{eq:trajectory3}
{dq \over dp} = -c-{f(p) \over q} ,
\end{equation}
on an open interval $(a,b)$ if it is negative on $(a,b)$, continuous on the closed interval $[a,b]$, and satisfies \eqref{eq:trajectory3} in the classical sense except at the discontinuity points of $f$.
\end{defi}

\begin{prop}\label{prop:trajectory} Let $f$ satisfy \eqref{H1}--\eqref{H3} and $p_u \in \{ \theta_{2i} \}_{i=1, \cdots, I}$. For any $c \in \R$, there exists a unique $p_l \in [0,p_u)$ and a unique function $q \in C^0 ([p_l,p_u])$ that solves \eqref{eq:trajectory3}, $q(p_u)=0$,
\begin{equation}\label{eqn3.6}
  q(p ) < 0 \quad\text{ for all }\quad p\in (p_l,p_u),
\end{equation}
$q (p_l) = 0$ if $p_l >0$, and $q(p_l)\le 0$ if $p_l = 0$.
\end{prop}

\begin{figure}[h]
	
	\centering
	\includegraphics[width=.5\linewidth]{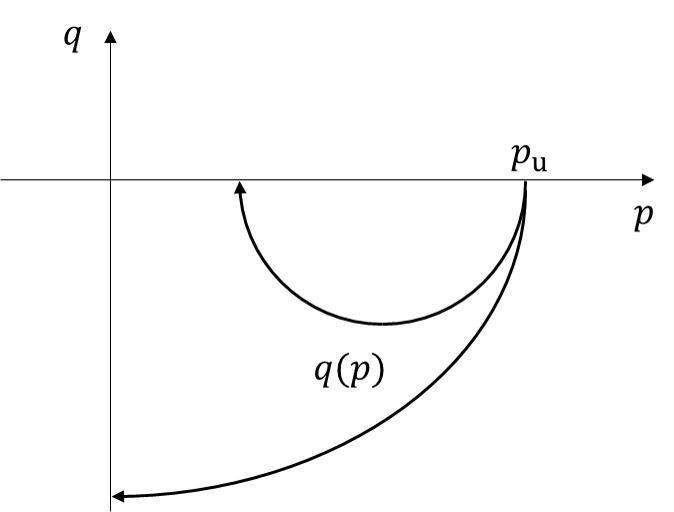}
	\caption{Two examples of trajectory $q(p)$ in the $pq$-phase plane. }
	\label{fig2}
\end{figure}

The proposition implies that the solution trajectory of $q=q(p)$ in the phase plane starts from $p=p_u\in \{ \theta_{2i} \}_{i=1, \cdots, I}$, stays in the fourth quadrant with $p>0$ and $q<0$, and is terminated when it touches one of the two axes $p=0$ or $q=0$ (see Figure~\ref{fig2}). The existence part of Proposition~\ref{prop:trajectory} was addressed in the above discussion. As a matter of fact, uniqueness follows from the same argument, by applying the classical Cauchy-Lipschitz theorem (as many times as the solution crosses a discontinuity point of the reaction function) to the ODE $\tilde{p}'' + c \tilde{p}' + \tilde{f} (p) =0$.

Going back to the original problem, the function $q$ in Proposition~\ref{prop:trajectory} gives a traveling wave; conversely, any (connected) traveling wave corresponds to a solution of~\eqref{eq:trajectory3}. Moreover, if $p_l$ is also a steady state of \eqref{eq:rd} and $q (p_l) = 0$, then this traveling wave monotonically connects $p_u$ and $p_l$ with speed $c$ (recall Definition~\ref{def:twsol}). We will first consider $p_u =1$ in order to construct the uppermost traveling wave of the propagating terrace. Then, at a later stage, we will make an iterative argument by setting $p_u$ as the lower bound of the previous traveling wave to build up the whole propagating terrace.

The next step is to see the dependency of the solution $q$ in Proposition \ref{prop:trajectory} on the wave speed $c$.
\begin{theo}[Monotonicity with respect to $c$] \label{monotonicity} Let $f$ satisfy \eqref{H1}--\eqref{H3} and $p_u \in \{ \theta_{2i} \}_{i=1, \cdots, I}$.
\begin{enumerate}[$(i)$]
\item Let $q_1$ and $q_2$ be two solutions of \eqref{eq:trajectory3} on an interval $(a,b)$ in the sense of Definition~\ref{defi:trajectory}, with $c$ replaced by respectively~$c_1$ and $c_2$. If $c_2 < c_1$ and $q_2 (b) \leq q_1 (b) < 0$, then $q_2 (p) < q_1 (p)$ for $a <p < b$.
\item Let $q_1$ and $q_2$ be the solutions of \eqref{eq:trajectory3} given in Proposition~\ref{prop:trajectory}, with $c$ replaced by respectively $c_1$ and $c_2$, and let $p_l^1$, $p_l^2$ be the lower bounds of their respective intervals of definition. If $c_2 < c_1$, then $p_l^2 \leq p_l^1$ (i.e., $p_l$ is an increasing function of $c$) and $q_2 (p) < q_1 (p)$ for $p_l^1 < p < 1$.
\end{enumerate}
\end{theo}

\begin{proof} We only deal with statement $(ii)$ since the proof of statement $(i)$ is almost identical. We introduce $w := q_1 - q_2$ which is well-defined on the interval $( p^{max}_l, p_u )$, where $p^{max}_l =\max \{ p^2_l, p^1_l \}$. From equation~\eqref{eq:trajectory3}, we have on $( p^{max}_l, p_u )$ that
	$$ {dw \over dp} = -(c_1 - c_2) + {f(p) \over q_1 q_2} w.$$
	Now we define
	$$\gamma(p) := \exp\Big(-\int_{p_0}^p {f(s) \over q_1(s) q_2(s)} ds \Big),$$
	where $p_0 \in ( p^{max}_l, p_u )$. First we check that $\gamma (p)$ is well-defined for $p \in [p_0,p_u]$. Let $\delta >0$ be such that $f(p) >0$ in $[p_u-\delta,p_u)$. Since $q_1$ and $q_2$ are continuous and negative on $[p_0,p_u-\delta]$, we get that $\max \{ q_1 , q_2 \} < -C_\delta$ for a positive constant $C_\delta>0$. Thus, recalling the upper bound $|f| \leq M$, we find that $\gamma(p) \le \exp({M \over C_\delta^2})$ for all $p \in [p_0, p_u-\delta]$. On the other hand, $f(p) >0$ and hence $\gamma '(p) \leq 0$ on $[p_u-\delta,p_u)$. Since $\gamma (p)$ is positive, it admits a limit as $p \to p_u$ and we can extend it by continuity at $p =p_u$. The resulting function $\gamma(p)$ is well-defined and bounded on~$[p_0,p_u]$.
	
	Now, we multiply the above equation on $w$ by $\gamma$, and find that
	$$ {d(\gamma w) \over dp} = -(c_1 - c_2) \gamma.$$
	Taking $p \in [p_0,p_u]$ and integrating from $p$ to $p_u$, we obtain
	$$ \gamma(p) w(p) = \gamma(p_u) w(p_u) + (c_1 - c_2) \int_p^{p_u} \gamma(s) ds.$$
	Since $\gamma(p_u)w(p_u) \ge 0$, $c_1-c_2 > 0$, and $ \gamma(p) >0$ for $p_0 \le p < p_u$, we infer that $w(p) > 0$ for $p_0 \le p < p_u$. Since we can choose $p_0$ arbitrary close to $p^{max}_l$, this implies that $q_2(p) < q_1(p)$ for $p \in (p^{max}_l,p_u)$. By continuity, we also have $q_2(p_l^{max}) \le q_1(p_l^{max})$.

Now assume by contradiction that $p_l^2 > p_l^1$. Then $p^{max}_l = p_l^2 > p_l^1 \ge 0$. In particular, from the definition of $p_l^2$, we have $q_2(p_l^2) = 0$. From the above, $q_1(p_l^2) \ge q_2(p_l^2) = 0$, thus $q_1(p_l^2)=0$ which in turn implies that $p_l^1 \ge p_l^2 $, a contradiction. This completes the proof.
\end{proof}

\begin{theo}[Continuity with respect to $c$]\label{th:continuity}
	Let $f$ satisfy \eqref{H1}--\eqref{H3} and $p_u \in \{ \theta_{2i} \}_{i=1, \cdots, I}$. Consider a sequence $(c_n)_{n \in \mathbb{N}}$ that converges to $c\in\R$ as $n \to \infty$. We define $p_l$, $q$ from Proposition~\ref{prop:trajectory}, as well as $p_{l,n}$, $q_{n}$ from the same proposition with $c_n$ instead of $c$. 
	\begin{enumerate}[$(i)$]
		\item If the sequence $(c_n)_{n \in \mathbb{N}}$ increases, then $(p_{l,n},p_u)\subset(p_l,p_u)$, $q_{n} \to q$ locally uniformly on the interval $(p_l,p_u)$, and $\lim_{n\to\infty} q_n (p_{l,n}) = q(p_l)$.
		\item If the sequence $(c_n)_{n \in \mathbb{N}}$ decreases, then $\lim_{n \to \infty} p_{l,n} = p_l$, $q_n \to q$ locally uniformly on the interval $(p_l,p_u)$, and $\lim_{n\to\infty} q_n (p_{l,n}) = q(p_l)$.
	\end{enumerate}
\end{theo}
\begin{proof}
	$(i)$ We first consider the case of a increasing sequence. From Theorem \ref{monotonicity}, we have $p_{l,n} \leq p_l$ for all $n \in \mathbb{N}$. Thus, $q_n$ and $q$ are well-defined on $[p_{l}, p_u]$. Let us then introduce $w_n := q - q_n$ which satisfies
		$$ {dw_n \over dp} = -(c - c_n) + {f(p) \over q q_n} w.$$
		As in the proof of Theorem~\ref{monotonicity}, we fix $p_0 \in (p_{l},p_u)$ and define
		$$\gamma_n (p) := \exp\Big(-\int_{p_0}^p {f(s) \over q(s) q_n(s)} ds \Big),$$
		which can be extended by continuity to the closed interval $[p_0,p_u]$. In particular, the function~$\gamma_n$ is bounded on the interval $[p_0,p_u]$. Let us now check that it is also uniformly bounded with respect to $n \in \mathbb{N}$. First, thanks to Theorem~\ref{monotonicity}, we have that $q_n \leq q <0$ in $[p_0,p_u)$, thus
		\begin{equation}\label{continuity_1}
		0 \leq \gamma_n (p) \leq  \exp\Big( \int_{p_0}^p {M \over |q| |q_n|} ds \Big) \le \exp\Big( \int_{p_0}^p {M \over q^2} ds \Big),
		\end{equation}
		for some $M>0$. Now let $\delta >0$ be such that $f >0$ on $[p_u-\delta,p_u)$. Since $q$ is negative and continuous in $[p_0, p_u -\delta]$, we get from \eqref{continuity_1} that $\gamma_n (p)$ is uniformly bounded with respect to $p \in [p_0, p_u-\delta]$ and $n \in \mathbb{N}$. On the other hand, we have that $\gamma_n (p)$ is decreasing in $[p_u-\delta,p_u]$. It easily follows that there exists $M_\gamma >0$ such that
		$$ 0 \leq \gamma_n (p) \leq M_\gamma,$$ for all $p \in [p_0,p_u]$ and $n \in \mathbb{N}$. Next, we multiply the above equation on $w_n$ by $\gamma_n$, and find that
		$$ {d(\gamma_n w_n) \over dp} = -(c - c_n) \gamma_n.$$
		Integrating from $p_0$ to $p_u$, we obtain
		$$ \gamma_n (p_0) w_n(p_0) = \gamma_n (p_u) w_n(p_u) + (c - c_n) \int_{p_0}^{p_u} \gamma_n(s) ds.$$
		From its definition, $\gamma_n(p_0) = 1$. Since $\gamma_n(p) \leq  M_\gamma$, and $w_n(p_u) = 0$, we get
		$$|w_n(p_0)| \le |c-c_n| \int_{p_0}^{p_u} |\gamma_n(s)| ds \le M_\gamma|c-c_n|$$
		Thus, $q_n(p_0) \to q(p_0)$ as $n \to \infty$. Since we chose $p_0$ arbitrarily in $(p_{l}, p_u)$, we have proved pointwise convergence in $(p_l,p_u)$ (and even in $(p_l,p_u]$ since $q_n (p_u) = q(p_u)= 0$ for all~$n$). Applying Dini's theorem, the convergence is also locally uniform in the same interval.
		
		Now, we will show that $\lim_{n \to \infty} q_n (p_l) = q(p_l)$. Since the functions $q$ and $q_n$ are nonpositive, it is enough to show that $\lim_{n \to \infty} q_n^2 (p_l) = q^2(p_l)$. Let us consider the equation satisfied by~$q_n^2$, which is obtained by multiplying~\eqref{eq:trajectory3} satisfied by~$q_n$ by $q_n$ itself. Define $\overline{c} = \max_{n\in \mathbb{N}} \left\{ |c|, |c_n| \right\}$ and notice that, by Theorem~\ref{monotonicity}, we have $q_n(p) > -B$ for all $n$ and $p\in (p_{l,n},p_u)$ where $-B$ is a lower bound of $q_0$. Recalling also that $M$ is an upper bound for $|f|$, we get that
		\begin{equation*}
		\left| {dq_n^2 \over dp} \right|= \left| -2c_n q_n - 2f(p) \right| < 2 \overline{c} B + 2M 	\quad \mbox{ for } p\in (p_l,p_u).
		\end{equation*}
It follows that the functions $\{ q^2 \} \cup \{ q_n^2\,  | \, n \in \mathbb{N} \}$ are uniformly Lipschitz continuous on the interval~$[p_l, p_u]$. By Arzela-Ascoli's theorem and the uniqueness of the limit, we get as wanted that $\lim_{n \to \infty} q_n^2 (p_l) = q^2 (p_l)$.

$(ii)$ Let us now consider a decreasing sequence $(c_n)_{n \in \mathbb{N}}$ such that $c_n \to c$ as $n \to \infty$. From Theorem~\ref{monotonicity}, we know that $p_{l,n}$ is a decreasing sequence which is bounded from below by $p_l \geq 0$. Thus we can define
$$p_{\infty} = \lim_{n \to \infty} p_{l,n} \in [p_l,p_u).$$
In particular, for any $p \in ( p_{\infty} , p_u)$, we can find $N$ large enough such that $p \in (p_{l,N},p_u)$. Then, for any $n_2 \geq n_1 > N$, $q_{n_1} (p )$ and $ q_{n_2} (p)$ are well defined and
	\begin{equation*}\label{eq:lem_continuity_1}
		0 > q_{n_1} (p ) \geq q_{n_2} (p) \geq q (p).
	\end{equation*}
Arguing as in the proof of statement~$(i)$, we can find that $q_n \to q \le 0$ locally uniformly in $(p_{\infty} , p_u]$.

It remains to show that $p_\infty = p_l$ and $\lim_{n \to \infty} q_n (p_{l,n}) = q (p_l)$. We first consider the case when $p_{l,n}=0$ for some $n$. Then, by Theorem~\ref{monotonicity}, we have $p_l= 0$ and $p_{l,n}= 0$ for any large $n$, so that $\lim_{n \to \infty} p_{l,n} = p_l$ is satisfied. Moreover, as in the proof of statement~$(i)$, we have that
		\begin{equation*}
		\left| {dq_n^2 \over dp} \right| = \left| -2c_n q_n - 2f(p) \right| < 2 \overline{c} B + 2M 	\quad \mbox{ for } p\in (0,p_u),
		\end{equation*}
		for any $n$ large enough, where $\overline{c} = \max_{n \in\mathbb{N}} \{ |c|,|c_n| \}$ and $B$ (resp. $M$) is an upper bound for $|q|$ (resp. $|f|$). Thus the sequence $\{ q_n^2\}$ is uniformly Lipschitz continuous, hence equicontinuous. From Arzela-Ascoli's theorem and uniqueness of the limit, we have that $q_n^2$ converges to $ q^2$ uniformly in $[0,p_u]$, and in particular $\lim_{n \to \infty} q_n (p_{l,n}) = \lim_{n \to \infty} q_n (0) = q(0)$.

	Now consider the case when $p_{l,n}>0$ for all $n$. Then, $q_n(p_{l,n}) = 0$ for all $n$ and $\lim_{n\to\infty} q_n (p_{l,n})=0$. It is enough to prove that $q (p_\infty) = 0$, so that in particular $p_l = p_\infty$. To do this, we again use the fact that
\begin{equation*}
		\left| {dq_n^2 \over dp} \right| = | -2c_n q_n - 2f(p)| < 2  \overline{c} B + 2M 	\quad \mbox{ for } p\in (p_{l,n},p_u).
\end{equation*}
Integrating from $p_{l,n}$ to any $p \in (p_{l,n},p_u)$ and recalling that $q_n (p_l{,n}) =0$, we get that
\begin{equation*}
		|q_n^2(p)| = |q_n^2(p)-q_n^2(p_{l,n})| \le (2 \overline{c} B + 2M)|p-p_{l,n}| .
		\end{equation*}
Passing to the limit as $n \to \infty$, it follows that
$$| q^2 (p) | \leq ( 2 \overline{c} B + 2M ) | p - p_\infty|,$$
for all $p \in (p_\infty,p_u)$, hence $ q^2 (p_\infty)= 0$ by continuity.
\end{proof}
We will use a continuity argument to find a $c$ such that the solution in Proposition \ref{prop:trajectory} provides a traveling wave solution connecting two stable steady states. To do so, we first consider the cases when~$|c|$ is too large, which is the purpose of the next two results.
\begin{lem}[Lower bound of traveling wave speeds]\label{th:-clarge}
	Let $f$ satisfy \eqref{H1}--\eqref{H3} and $p_u \in \{ \theta_{2i} \}_{i=1, \cdots, I}$.  There exists $ C_1\in\R$ such that the solution trajectory touches the vertical $q$-axis before touching the horizontal $p$-axis for all $c \leq C_1$. In other words, $p_l = 0$ and $q (0) < 0$.
\end{lem}
\begin{proof}
	Assume that $p_l >0$, so that $q(p_l) =0$. Multiply~\eqref{eq:trajectory3} by $q$ and integrate it from $p_l$ to~$p_u$. Then we obtain
	$${q(p_u)^2 - q(p_l)^2 \over 2} = -c\int_{p_l}^{p_u} q dp - \int_{p_l}^{p_u} f(p) dp = 0,$$
	since $q(p_l) = q(p_u) = 0$. 	On the other hand, we have that $\int_{p_l}^{p_u} q dp < 0$, and by Theorem~\ref{monotonicity} this is also an increasing function of the parameter~$c$. In particular,
$$- c \int_{p_l}^{p_u} q dp \to - \infty,$$
as $c \to - \infty$. Thus, the above equality clearly leads to a contradiction if $c$ is small enough.
\end{proof}

\begin{lem}[Minimum monostable traveling wave speed]\label{th:clarge} Let $f$ satisfy \eqref{H1}--\eqref{H3} and $p_u=\theta_{2i_0}$ for some $i_0\in\{1,\cdots,I\}$. Then, there exists $C_2>0$ such that $c \geq C_2$ if and only if the solution of \eqref{eq:trajectory3} satisfies that $p_l = \theta_{2i_0 -1} >0$, and hence $q (p_l) =0$.
\end{lem}
\begin{proof}
	Since $f(p_u) = 0$ and $f$ is $C^1$ on $[\theta_{2i_0 -1},p_u)$, we can choose $K >0$ large enough so that
		$$\overline{f}(p) := K(p-\theta_{2i_0 -1}) \geq f(p),$$
	for all $p \in [\theta_{2i_0 -1},p_u]$. Now we take $c =2 \sqrt{K}$ and consider the ODE system
	\begin{eqnarray}
	\nonumber \overline{p}' &=& \overline{q}, \\ \nonumber \overline{q}' &=& -c\overline{q}-\overline{f}(\overline{p}).
	\end{eqnarray}
	Then the trajectory $(\overline{p},\overline{q})$ starting from $(p_u, - \sqrt{K} (p_u -\theta_{2 i_0 -1}) )$ of this ODE system is a straight line which converges to $(\theta_{2i_0 -1},0)$. As before, we can rewrite it as a function $\overline{q}$ of $p \in [\theta_{2i_0 -1},p_u]$, which satisfies
	\begin{equation} \label{eq:trajectory2}
	{1\over 2} {d\overline{q}^2 \over dp} = -c\overline{q}-\overline{f}(p),
	\end{equation}on $(\theta_{2 i_0 -1},p_u)$. Now we claim that
		\begin{equation}\label{cla:trajsuper}
		\forall p \in (\theta_{2i_0 -1}, p_u], \quad q (p ) > \overline{q} (p).
		\end{equation}
First, we point out that $q$ is indeed well-defined on $[\theta_{2i_0 -1},p_u]$, i.e. $p_l \leq \theta_{2i_0 -1}$, due to the fact that $f$ is positive between $\theta_{2i_0 -1}$ and $p_u$ and thus the trajectory of $q$ cannot touch the horizontal axis in the open interval.

Then, by construction, we have that $\overline{q} (p_u) < 0 = q(p_u)$. Let us proceed by contradiction and assume that there is some $p_0 > \theta_{2i_0-1}$ such that $\overline{q} (p) < q (p)$ for all $p \in (p_0,p_u]$ but $\overline{q} (p_0) = q (p_0)$. Multiplying \eqref{eq:trajectory3} by $q$ and substracting it from \eqref{eq:trajectory2}, we get the following equation:
	\begin{equation} \nonumber
	{1\over 2} {d(\overline{q}^2 - q^2) \over dp} = -c(\overline{q}-q)-(\overline{f}(p)-f(p)).
	\end{equation}
Integrating from $p_0$ to $p_u$ and recalling that $c= 2\sqrt{K}$, we get
	\begin{equation*}
	0 < \frac{\overline{q}^2 (p_u) - q^2 (p_u) }{2} = -c \int_{p_0}^{p_u} (\overline{q}-q) dp - \int_{p_0}^{p_u} (\overline{f}(p)-f(p)) dp < 0.
	\end{equation*}
	This is a contradiction and the claim~\eqref{cla:trajsuper} is proved. Finally, since $\overline{q} (\theta_{2i_0 -1}) = 0$ and $q$ cannot take the value 0 in the interval $(\theta_{2i_0 -1},p_u)$, we conclude that $p_l = \theta_{2 i_0 -1}$ and $q (\theta_{2i_0 -1}) = 0$ when $c = 2 \sqrt{K}$. Furthermore, we point out that $K$ can be increased without loss of generality so that the same conclusion holds for all $c$ large enough. Therefore, there exists $C>0$ such that $p_l=\theta_{2i_0-1}$ for all $c>C$. By the monotonicity with respect to $c$, Theorem \ref{monotonicity}, we may take the smallest such $C$ denoted by $C_2$.
\end{proof}

Now we are ready to prove the existence of a traveling wave connecting 1 (or any positive stable steady state) and some intermediate stable steady state $\theta_{2i}$ ($i \in 0 , \cdots I-1$). The terrace will be obtained by iterating this argument, which is why we state the following theorem with any stable steady state $p_u$ instead of 1.

\begin{theo}[Minimum bistable traveling wave speed]\label{th:upper_wave}
Let $f$ satisfy \eqref{H1}--\eqref{H3} and $p_u \in \{ \theta_{2i} \}_{i=1, \cdots, I}$. For any $c \in \mathbb{R}$, denote by $q (p; c)$ the solution of \eqref{eq:trajectory3}, and by $p_l (c)$ the lower bound of the solution domain given by Proposition \ref{prop:trajectory}. Denote by $c^*$ the maximal wave speed such that the solution does not touch the p-axis, i.e.
\[
c^* := \sup\left\{ c \ | \  p_l (c) = 0 \mbox{ and } q (p_l (c);c) < 0\right\}.
\]
Then, the function $q (\cdot;c^*)$ is a connected traveling wave, which monotonically connects $p_u$ and another stable steady state $p^*$ with $p^* < p_u$.
\end{theo}
\begin{proof}
	Notice that $c^*$ is a well-defined real number thanks to Lemmas~\ref{th:-clarge} and~\ref{th:clarge}. Let us now prove the theorem. For simplicity, we denote $p^* = p_l (c^*)$ in this proof.

Let us first show that $p^*$ is a steady state, i.e.
\begin{equation}\label{upper_wave_continuous}
\text{either} \ \ \ p^* \in \{ \theta_{2i} \, : \, i \in \{ 0 ,\cdots I \}  \} \ \  \text{ or } \ \  f ( p^* ) = 0.
\end{equation}
We recall here that, due to the discontinuities of the reaction $f$ at the stable steady states, it is not necessary to assume that $f (\theta_{2i}) = 0$ for $i =0,\cdots ,I$.

If \eqref{upper_wave_continuous} does not hold, then $p^* >0$ and also $f (p^*) < 0$. Indeed, it is clear from \eqref{eq:trajectory3} that $q (\cdot; c^*)$ cannot touch the horizontal axis at a point where $f$ is positive. Thus, the function~$f$ is Lipschitz-continuous on a neighborhood of $p^*$ and we can go back to and understand the ODE system~\eqref{phaseplane1}--\eqref{phaseplane2} with $c=c^*$ in the classical sense on some neighborhood of $X = (p^*, 0)$. In particular, the trajectory corresponding to $q (\cdot; c^*)$ enters the upper half plane $\{ q >0\}$ by passing through the point $X$. By continuity in the standard ODE theory, one can find $\delta >0$ and $\epsilon >0$ small enough such that any trajectory of~\eqref{phaseplane1}--\eqref{phaseplane2}, with $c= c^* -\delta$ and originating from $(p,q) \in B_\epsilon (X)$ the ball of radius $\epsilon$ centered at $X$, also crosses the horizontal axis. In terms of~\eqref{eq:trajectory3}, this provides a solution $\tilde{q}$ of \eqref{eq:trajectory3} with $c < c^*$ on some interval $(\tilde{p}, p^*+\epsilon_1)$ with $\epsilon_1 \in (0, \epsilon)$ such that $\tilde{q} (\tilde{p}) = 0$ and $\tilde{q} ( p^* + \epsilon_1 ) < q (p^* + \epsilon_1 ; c^*)$. 

Now take an increasing sequence $(c_n)_{n \in \mathbb{N}}$ such that $c_n \to c^*$ as $n \to \infty$. From statement~$(ii)$ of Theorem~\ref{monotonicity} and by our choice of $c^*$, we have that $p_l (c_n) = 0$ and $q (p_l (c_n);c_n) < 0$ for all $n \in \mathbb{N}$. By Theorem~\ref{th:continuity}, we also have that $q (p^* + \epsilon_1; c_n)$ converges to $q (p^* +\epsilon_1 ;c^*)$ as $n \to \infty$. Then, applying statement~$(i)$ of Theorem~\ref{monotonicity} to $q (\cdot; c_n)$ and $\tilde{q}$, we get for any~$n$ large enough that $q ( \cdot; c_n) \geq \tilde{q}$ and thus touches the horizontal axis. We have reached a contradiction and proved~\eqref{upper_wave_continuous}.\medskip

The next step is to show that
\begin{equation}\label{upper_wave_step2}
q (p^*; c^*) = 0.
\end{equation}
We take a decreasing sequence $(c_n)_{n \in \mathbb{N}}$ such that $c_n \to c^*$ as $n \to \infty$, and by our choice of~$c^*$, the function $q (\cdot; c_n)$ touches the horizontal axis at $p_l (c_n) \geq 0$ for any $n$. However, from Theorem~\ref{th:continuity}, we get as wanted that
$$\lim_{n \to \infty} q (p_l (c_n) ;c_n) = q (p^* ;c^*) = 0.$$

From \eqref{upper_wave_continuous} and \eqref{upper_wave_step2}, we now know that $p^*$ is a steady state and that $q (\cdot; c^*)$ defines a traveling wave with speed $c^*$ monotonically connecting $p_u$ and $p^*$. By construction $q(\cdot;c^*)$ is negative on the open interval $(p^*,p_u)$, thus this traveling wave is also connected in the sense of Definition~\ref{connectedcompact}. It only remains to check that $p^*$ is a stable steady state.

We proceed again by contradiction, and assume that $p^*$ is one of the unstable steady states $\theta_{2i+1}$ with $i \in \{0,\cdots,I-1 \}$. Let us first check that $c^* >0$. Multiplying \eqref{eq:trajectory3} by $q(\cdot; c^*)$ and integrating from $p \in [ p^*,p_u)$ to $p_u$, one obtains that
$$c^* \int_{p}^{p_u} q(s, c^*) ds =  \frac{ q(p , c^*)^2}{2} -  \int_{p}^{p_u} f(s) ds .$$
Since $p^*$ is an unstable steady state, it must be positive and thus $q (p^* ,c^*) =0$; moroever, by~\eqref{H1} the function $f$ is positive in an interval $(p^*, p^* + \delta)$ with $\delta >0$. It follows that the right hand term of the above equality is increasing on the same interval $(p^*,p^* + \delta)$. Due to the negativity of the function $q$ in $(p_l, p_u)$, there must hold that $c^* > 0$.

The argument is now the same as in the first step above. Due to \eqref{H2}, the function~$f$ is Lipschitz-continuous on a neighbordhood of $p^*$. Going back to the original ODE system, the function $(p,q (\cdot ; c^*))$ defines a solution of~\eqref{phaseplane1}-\eqref{phaseplane2} with $c= c^*$ which converges to $X = (p^*, 0)$ at $\infty$. Since $c^* >0$, the equilibrium point $X$ is also a stable (either node or spiral) point for any $c$ close enough to $c^*$. It follows from the standard ODE theory that there exist $\delta>0$ and $\epsilon >0$ small enough so that the solution of~\eqref{phaseplane1}--\eqref{phaseplane2} with $c =c^* - \delta$, starting from any $(p,q) \in B_\epsilon(X)$, also converges to the equilibrium point~$X$ while remaining in $B_\epsilon (X)$. This provides $\epsilon_1 \in (0, \epsilon)$ and a solution~$\tilde{q}$ of~\eqref{eq:trajectory3} with some $c < c^*$, on an interval $(\tilde{p}, p^* + \epsilon_1)$ with $\tilde{p} >0$, and such that $\tilde{q} (\tilde{p}) = 0$ and $\tilde{q} (p^* + \epsilon_1) < q (p^* + \epsilon_1; c^*)$. Notice that $\tilde{p}$ may or may not be equal to $p^*$, depending on whether we are in the spiral or node case.

Regardless, we consider an increasing sequence $(c_n)_{n \in \mathbb{N}}$ such that $c_n \to c^*$ as $n \to \infty$, as well as $p_l (c_n) = 0$ and $q (p_l (c_n);c_n) < 0$ for all $n \in \mathbb{N}$. By Theorem~\ref{th:continuity} and statement~$(i)$ of Theorem~\ref{monotonicity}, we find that $q (\cdot ;c_n) \geq \tilde{q}$ and thus it touches the horizontal axis for~$n$ large enough, a contradiction.
\end{proof}
		
We are almost ready to prove the existence of a propagating terrace. First, we choose $p_u = 1$ and obtain a traveling wave connecting $1$ to another stable steady state $p^* < 1$ by Theorem~\ref{th:upper_wave}. If $p^* =0$, then we have already found a propagating terrace connecting 1 and 0 which consists of a single front. When $p^* >0$, one may apply Theorem~\ref{th:upper_wave} again with $p_u = p^*$ to find a second traveling wave, and repeat the process until one reaches the lowest stable steady state 0. This obviously happens in a finite number of steps since there are finite number of stable steady states between $1$ and $0$. However, a key step is missing because the wave speeds of the obtained traveling waves a priori may be ordered incorrectly. Hence, this sequence is not a terrace yet. In order to fill this gap, we should clarify the cases with two or more traveling waves of the same speed $c\in\R$. The $p_l$ in Proposition \ref{prop:trajectory} is the first contact point $p_l$. If the strict inequality in \eqref{eqn3.6} is replaced with non-strict one, other contact points can be included. Note that, in Theorem \ref{th:continuity}, $p_{l,n}$ converges to $p_l$ only when $c_n$ decreases to $c\in\R$ as $n\to\infty$. If $c_n$ increases as $n\to\infty$, the convergence fails in general. In other words, if we consider $p_l$ as a function of wave speed $c$, it is right continuous, not left continuous. The final step to obtain the existence of a terrace is to understand the situation that $\lim_{n\to\infty}p_{l,n}\ne p_l$ when $c_n$ increases to $c$ as $n\to\infty$.

\begin{theo}[Continuity beyond $p_l$]\label{th:beyond}
Let $f$ satisfy \eqref{H1}--\eqref{H3} and $p_u \in \{ \theta_{2i} \}_{i=1, \cdots, I}$. Let $(c_n)_{n \in \mathbb{N}}$ be an increasing sequence such that $c_n \to c\in\R$ as $n \to \infty$. Let $(p_l,q)$ and $(p_{l,n},q_n)$ be the ones in Proposition~\ref{prop:trajectory} when wave speeds are $c$ and $c_n$, respectively. Suppose that $\lim_{n \to\infty} p_{l,n}\ne p_l$. Then the following two statements hold.\\
$(i)$ This situation is possible only when $p_l \in \{ \theta_{2i} \}_{i=1, \cdots, I-1}$, i.e., a positive stable steady state.\\
$(ii)$ The sequence $q_n$ converges to a continuous function $q_\infty$ which solves \eqref{eq:trajectory3} on some nontrivial intervals $(\widehat{p}_{j+1} ,\widehat{p}_{j})$, with $\widehat{p}_M =\lim_{n \to\infty} p_{l,n}$ and $\widehat{p}_0 = p_l $ for some $M \in \mathbb{N}^*$. Moreover, the $\widehat{p}_j \in \{ \theta_{2i} \}_{i=0, \cdots, I-1}$ are stable steady states and $q_\infty (\widehat{p}_j)=0$ for any $j=0,\cdots, M-1$. In particular, the restriction of $q_\infty$ to $[\widehat{p}_{j+1},\widehat{p}_{j}]$ is the solution given by Proposition~\ref{prop:trajectory} with $p_u = \widehat{p}_{j}$.
\end{theo}
\begin{proof}
First, we already know from Theorem~\ref{th:continuity} that $q_n$ converges locally uniformly to $q$ on the interval $(p_l,p_u)$. Since $p_l > \lim_{n \to\infty} p_{l,n} \geq 0$, one can proceed as in the proof of Theorem~\ref{th:upper_wave} to find that $p_l$ must be a stable steady state; we omit the details. In particular, one can apply Proposition~\ref{prop:trajectory} with $p_u$ replaced by $p_l$ to find a solution $\widehat{q}_0$ solving \eqref{eq:trajectory3} with speed $c$ on some interval $(\widehat{p}_1,p_l)$, together with $\widehat{q} (p_l) = 0$ and either $\widehat{q} (\widehat{p}_1) = 0$ or $\widehat{p}_1 =0$.

Next, it follows from Theorem~\ref{monotonicity} that $\lim_{n \to \infty} p_{l,n} \leq \widehat{p}_1$ and $q_n \leq \widehat{q}_1$ on $[\widehat{p}_1, p_l]$. Indeed, we have that $q_n (p_l) < 0 = \widehat{q}_1 (p_l)$ for any $n \in \mathbb{N}$. Thus, for any $n$ and any $\delta >0$ small enough (possibly depending on $n$), we have $q_n (p_l -\delta) < \widehat{q}_1 (p_l - \delta) < 0$. Then statement~$(i)$ of Theorem~\ref{monotonicity} applies and one finds that $q_n (p) < \widehat{q}_1 (p)$ for any $p \in (\widehat{p}_1, p_l - \delta)$, so that in particular $\widehat{p}_1 \geq \lim_{n \to \infty} p_{l,n}$. Since $\delta$ is arbitrarily small, we also infer that $q_n (p) < \widehat{q}_1 (p)$ for any $n \in \mathbb{N}$ and $p \in (\widehat{p}_1, p_l)$, hence $q_n (p) \leq \widehat{q}_1 (p)$ in the closed interval by continuity.

We now know that
$$q_0 \leq q_n \leq \widehat{q}_1 <0 \quad \mbox{ for } p \in (\widehat{p}_1,p_l),$$
and also, from Theorem~\ref{th:continuity} and the fact that $p_l>0$, that $\lim_{n \to \infty} q_n (p_l) = 0$. Then, proceeding again as in the proof of Theorem~\ref{th:continuity}, one can check that $q_n$ converges to $\widehat{q}_1$ locally uniformly in $(\widehat{p}_l, p_l]$ and that $\lim_{n \to \infty} q_n ( \widehat{p}_1) = \widehat{q}_1 (\widehat{p}_1)$.

If $\widehat{p}_1 = \lim_{n \to \infty} p_{l,n}$, then $M=1$ and Theorem~\ref{th:beyond} is already proved. In the other case when $\widehat{p}_1 > \lim_{n\to \infty} p_{l,n} \geq 0$, then one may check that $\widehat{p}_1$  is a stable steady state. We again omit the details since the argument is the same as in the proof of Theorem~\ref{th:upper_wave}. One can then reiterate the above argument and find that $q_n$ converges to the solution provided by Proposition~\ref{prop:trajectory} with $p_u$ replaced by $\widehat{p}_1$, on the interval $(\widehat{p}_2, \widehat{p}_1)$ where $\widehat{p}_2$ is the lower bound of the interval of definition of this solution. Again, $\widehat{p}_2$  is either equal to $\lim_{n \to \infty} p_{l,n}$ or it is a stable steady state. We reiterate until we reach $\widehat{p}_M = \lim_{n \to \infty} p_{l,n}$ for some integer~$M$, which happens in a finite number of steps because there are finitely many stable steady states. Theorem~\ref{th:beyond} is proved.
\end{proof}

We are ready to prove the existence part of Theorem \ref{th:exist}.
\begin{theo} \label{exist}
	 There exists at least one propagating terrace for \eqref{eq:rd} connecting 1 and 0, in the sense of Definition~\ref{terracedef}. Moreover, the terrace satisfies Theorem~\ref{th:exist} $(i)$ and $(ii)$.
\end{theo}
\begin{proof}
We first construct the terrace by iteration. We denote, for any $c \in \R$, by $q (\cdot; c)$ the solution from Proposition~\ref{prop:trajectory} with $p_u =1$, and by $p_l (c)$ the lower bound of its interval of definition. According to Theorem~\ref{th:upper_wave}, we have that $q (\cdot ; c_1^*)$ defines a connected traveling wave, monotonically connecting~1 and some stable steady state~$p_l (c_1^*)$ with speed
$$c_1^* = \sup\left\{ c \ | \  p_l (c) = 0 \right\}.$$
If $p_l (c_1^*) = 0$ then this is also a propagating terrace consisting of only one traveling wave. Now consider the case when $p_l (c_1^*)>0$.

We take an increasing sequence $(c_n)_{n \in \mathbb{N}}$ such that $c_n \to c^*_1$ as $n \to \infty$. From the definition of $c^*_1$ we have that $\lim_{n \to \infty}  p_l (c_n) = 0 < p_l (c_1^*)$. Therefore we can apply Theorem~\ref{th:beyond} and we get that $q (\cdot; c_n)$ converges to a function $q_\infty$ which solves~\eqref{eq:trajectory3} with speed $c= c_1^*$ on every subintervals $(\widehat{p}_{j+1} ,\widehat{p}_{j})$, with $0 = \widehat{p}_M < \cdots < \widehat{p}_0 = p_l (c_1^*) $ for some $M \in \mathbb{N}^*$. Moreover, the $\widehat{p}_j \in \{ \theta_{2i} \}_{i=0, \cdots, I-1}$ are stable steady states and $q_\infty (\widehat{p}_j)=0$ for any $j=0,\cdots, M-1$. If $q_\infty (0)$ is also equal to 0, then this defines a finite sequence of connected traveling fronts, monotonically connecting $\widehat{p}_j$ and $\widehat{p}_{j+1}$, with same speed $c_1^*$. In such a case, we have also found a propagating terrace connecting 1 and 0.

If instead $q_\infty (0) <0$, then we have obtained a propagating terrace connecting 1 and the positive stable steady state $\widehat{p}_{M-1}$, whose fronts all have the same speed $c^*_1$. 
In order to find the next traveling wave of the propagating terrace, we take $p_u = \widehat{p}_{M-1}$ in Theorem~\ref{th:upper_wave} and get a traveling wave with speed $c_2^*$ connecting $\widehat{p}_{M-1}$ and some lower stable steady state~$p_l (c_2^*)$. It is given by $\tilde{q} (\cdot ; c_2^*)$ where $\tilde{q} (\cdot; c)$ is the solution from Proposition~\ref{prop:trajectory} with $p_u = \widehat{p}_{M-1}$, defined on an interval $(\tilde{p}_l (c), \widehat{p}_{M-1})$ and satisfying $\tilde{q} (\widehat{p}_{M-1};c)=0$ together with either $\tilde{q} ( \tilde{p}_l (c) ; c) = 0$ or $\tilde{p}_l (c)=0$. Moreover,
$$c_2^* := \sup\left\{ c \ | \  \tilde{p}_l (c) = 0 \mbox{ and } \tilde{q} ({p}_l (c);c) < 0\right\} .$$
Now recall from Theorem~\ref{th:beyond} that $q_\infty (0) <0$ and that $q_\infty$ coincides with $\tilde{q} (\cdot; c_1^*)$ on $(0, \widehat{p}_{M-1})$. Therefore
$$c_2^* >  c_1^* .$$
Putting the propagating terrace connecting $1$ and $\widehat{p}_{M-1}$ (with speed $c_1^*$) together with the traveling wave connecting $\widehat{p}_{M-1}$ and $\tilde{p}_l (c_2^*)$ (with speed $c_2^*$), we obtain a propagating terrace connecting~$1$ and $\tilde{p}_l (c_2^*)$.

If again~$\tilde{p}_l (c_2^*)>0$, then one can reiterate the above argument until reaching $0$. This iteration ends in a finite number of steps since there is only finitely many stable steady states. One finally obtains a propagating terrace connecting 1 and 0.\medskip

It now remains to check that this propagating terrace satisfies statements~$(i)$ and $(ii)$ of Theorem~\ref{th:exist}. By construction, the traveling waves of the propagating terrace are associated with negative solutions of \eqref{eq:trajectory3}, and therefore they are decreasing and connected in the sense of Definition~\ref{connectedcompact}. Moreover, we have already established that all the platforms are stable steady states.

Finally, these traveling waves are compact in the sense of Definition~\ref{connectedcompact}. This follows from the fact that the platforms are stable steady states, at which the reaction function~$f$ is discontinuous. Indeed, consider any traveling wave $\phi$ monotonically connecting two stable steady states $\theta_i$ and $\theta_j$ with $i >j$. Let us consider the right limit and prove that there must exist $Z < + \infty$ such that $\phi (Z) = \theta_j$ (the left side can be handled by a symmetrical argument). If there does not exist such a finite $Z$, then $\phi > \theta_j$ on the whole real line, and in particular it solves
$$\phi '' + c \phi' + \tilde{f} (\phi )= 0,$$
on a right half-line, together with
$$\lim_{z \to \infty} \phi (z) = \theta_,$$
where $\tilde{f}$ a $C^1$-function which coincides with $f$ on some interval $(\theta_j, \theta_j +\delta]$ with $\delta >0$ small. This is impossible because $\theta_j$ is not a zero of $\tilde{f}$ and thus such a solution does not exist. This concludes the proof of Theorem~\ref{th:exist}.
\end{proof}

We briefly highlight the fact that we have constructed a propagating terrace in the sense of Definition~\ref{terracedef}. As we pointed out in the introduction, the existence of a terrace solution in the sense of Definition~\ref{terracesolution} also follows thanks to the fact that this propagating terrace consists of compact traveling waves.

\section{Uniqueness of terrace}\label{sec:unique}

We have constructed a propagating terrace in Section~\ref{sec:exist}. In this section, we show the uniqueness of a terrace and complete the proof of Theorem \ref{th:exist}. The argument relies on the properties of solutions of \eqref{eq:trajectory3}. First, notice that, according to Definitions~\ref{def:twsol} and~\ref{connectedcompact}, if $\phi$ is a connected traveling wave monotonically connecting two steady states $b  > a$, then one can use the change of variable
$$p = \phi_i^{-1} (z)$$ to get a function $q$
solving
	$$ {dq \over dp} = -c - {f(p) \over q},$$
that is \eqref{eq:trajectory3}, in the inverval $(a,b)$ together with
$$q (a) = q (b) = 0.$$
For convenience, we will refer to this function $q$ as the trajectory of the traveling wave~$\phi$, which is consistent with the fact that the curve of the function $q$ is indeed the trajectory of the solution~$\phi$ in the phase plane of the ODE~\eqref{TWeqn}

In particular, we may rewrite some of the results in Section~\ref{sec:exist} in terms of the traveling waves, which is the purpose of the next two lemmas:
\begin{lem}\label{rewrite1} Let $b$ be a stable steady state. Let also $\phi_1$ and $\phi_2$ be two traveling waves monotonically connecting respectively $b$ and $a_1$ with speed $c_1$, $b$ and $a_2$ with speed $c_2$.
\begin{enumerate}[$(i)$]
\item If $c_1 \geq  c_2$, then $a_1 \geq a_2$.
\item If
$$c_2 = c^* =  \sup\left\{ c \ | \  p_l (c) = 0 \mbox{ and } q (p_l (c);c) < 0\right\},$$
then $c_1 \geq c_2$ and $a_1 \geq a_2$. Here $q (p;c)$ is the solution of \eqref{eq:trajectory3} from Proposition~\ref{prop:trajectory} with $p_u = b$, and $p_l (c)$ the lower bound of its interval of definition.
\end{enumerate}
\end{lem}
\begin{proof}
When $c_1 = c_2$, statement~$(i)$ simply follows from the uniqueness of the solution in Proposition~\ref{prop:trajectory}, which insures that $\phi_1$ and $\phi_2$ have the same trajectory $q$, thus they must coincide up to a shift. When $c_1 > c_2$, then it instead follows from Proposition~\ref{prop:trajectory} and Theorem~\ref{monotonicity}, noticing that $a_1$ and $a_2$ must coincide with $p^1_l$ and $p^2_l$. Statement $(ii)$ is also a consequence of Theorem~\ref{monotonicity}, which insures that the solution from Proposition~\ref{prop:trajectory} always crosses the vertical axis below the origin when $c < c^*$, and thus it cannot be the trajectory of a traveling wave monotonically connecting steady states.
\end{proof}
\begin{lem}\label{rewrite2} Let $\phi_1$ and $\phi_2$ be two traveling waves monotonically connecting respectively $b_1$ and $a_1$ with speed $c_1$, $b_2$ and $a_2$ with speed $c_2$. Denote also by $q_1$ and $q_2$ their respective trajectories.

If moreover $c_1 \geq  c_2$ and $b_1 \in (a_2, b_2)$, then $a_1 \geq a_2$ and $q_1 > q_2$ in the interval $(a_1, b_1)$.
\end{lem}
\begin{proof}
This follows from applying statement~$(i)$ of Theorem~\ref{monotonicity} to the trajectories $q_1$ and $q_2$ on the interval $( \max \{ a_1 , a_2\} , b_1 - \delta)$ with $\delta >0$ arbitrarily small. Notice indeed that $q_1 (b _1) = 0 > q_2 (b_1)$ so that the hypotheses of statement~$(i)$ of Theorem~\ref{monotonicity} are satisfied for any small enough~$\delta$. We get that $q_2 < q_1$ on the interval $( \max \{ a_1 , a_2\} , b_1 )$, which together with the facts that $q_1$ is negative on $(a_1,b_1)$ and $q_2 (a_2) = 0$  in turn insures that $\max \{ a_1, a_2 \} = a_1$.
\end{proof}
We will also need the next lemma, which as a matter of fact is a byproduct of the proof of Theorem~\ref{monotonicity}.
\begin{lem}\label{speedunique}
	Let $\phi_1$ and $\phi_2$ be two connected traveling waves monotonically connecting two stable steady states $b>a$, respectively with speeds $c_1$ and $c_2$.

Then $c_1 = c_2$ and $\phi_1 \equiv \phi_2 (\cdot + Z)$ for some shift $Z \in \mathbb{R}$.
\end{lem}
\begin{proof}
	As explained above, according to Definitions~\ref{def:twsol} and~\ref{connectedcompact}, the functions $\phi_1$ and $\phi_2$ are invertible respectively in the supports of $\phi_1 ' $ and $\phi_2 '$. Using the change of variables $p = \phi_i^{-1} (z)$ to rewrite the ODEs satisfied by these traveling waves, one finds functions $q_1$ and~$q_2$ solving respectively
	$$ {dq_1 \over dp} = -c_1 - {f(p) \over q_1},$$
	$$ {dq_2 \over dp} = -c_2 - {f(p) \over q_2},$$
in the sense of Definition~\ref{defi:trajectory} on the interval $(a,b)$, together with
		$$q_1 (a) = q_2 (b) = 0.$$
	As in the proof of Theorem~\ref{monotonicity}, we define $w=q_1 - q_2$ and subtract the above two equations. Then, we get
	\begin{equation}\label{eq:ww}
	{dw \over dp} - {f(p) \over q_1 q_2} w = -(c_1 - c_2).
	\end{equation}
	Fix $p_0 \in (a, b)$ and for $p \in [a, b]$, define the integrating factor $\gamma (p)$ as
	\begin{equation*}
	\label{integrating factor}
	\gamma (p) := \exp\left( - \int_{p_0}^p {f(s) \over q_1(s)q_2(s)} ds \right).
	\end{equation*}
	Note that it is well-defined and continuous on $(a,b)$. It is also decreasing with respect to $p$ on a left neighborhood of $b$, and increasing in a right neighborhood of $a$, due \eqref{H1} and the fact that $a,b \in \{ \theta_{2i} \}_{0 \leq i \leq I}$ the set of stable steady states. In particular $\gamma$ is bounded on the close interval $[a,b]$.

Next, multiplying~\eqref{eq:ww} by $\gamma$, one gets that
	\begin{equation}\label{eq:gammagamma}
	{d(\gamma w) \over dp} = -(c_1 - c_2)\gamma.
	\end{equation}
Now recall that $(\gamma w)(a) = (\gamma w)(b) = 0$ by assumption, thus by Rolle's theorem there exists $p_1 \in (a,b)$ such that ${d(\gamma w) \over dp}(p_1) = 0$. Since $\gamma (p_1) > 0$, we get from \eqref{eq:gammagamma} that $c_1 = c_2$. Finally, it follows from Proposition~\ref{prop:trajectory} that $q_1 \equiv q_2$, which in turn implies that $\phi_1$ and $\phi_2$ are identical up to some shift.
\end{proof}
We are now in a position to prove the uniqueness of the propagating terrace. Hereafter we denote by $\mathcal{T}$ the propagating terrace constructed in Section~\ref{sec:exist}, by $(\phi_j)_{1 \leq j \leq J}$ the corresponding sequence of connected traveling waves with speeds $(c_j)_{1 \leq j \leq J}$, and by $(p_j)_{1 \leq j \leq J}$ its platforms, such that
$$1 = p_0 > p_ 1 > \cdots > p_J = 0.$$
We also let $\widehat{\mathcal{T}}$ denote another propagating terrace, $(\widehat{\phi}_j)_{1 \leq j \leq \widehat{J}}$ be the corresponding traveling waves with speeds $(\widehat{c}_j)_{1 \leq j \leq \widehat{J}}$, and $(\widehat{p}_j)_{0 \leq j \leq \widehat{J}}$ be its platforms.

Our goal is now to show that $\widehat{\mathcal{T}}$ actually coincides with $\mathcal{T}$, i.e. that they have the same platforms and that both families of traveling waves coincide up to some shifts.

\begin{prop} \label{uniqueness1}
The set of platforms of $\mathcal{T}$ is included in the set of platforms of $\widehat{\mathcal{T}}$.

In particular, for any $j \in \{1 , \cdots, J \}$, there exists $\widehat{j} \in \{ 1, \cdots , \widehat{J} \}$ such that the traveling wave~$\widehat{\phi}_j$ connects $p_{j-1} = p_{\widehat{j}-1}$ and $p_{\widehat{j}}$. Furthermore we have that $\widehat{c}_{\widehat{j}} \geq c_{j}$.
\end{prop}
\begin{proof}
	We only show that $p_1$ the uppermost platform (excluding $1$) of $\mathcal{T}$ belongs to the set of platforms of $\widehat{\mathcal{T}}$, and that $\widehat{c}_1 \geq c_1$. The result then follows by iteration.

Recall that $\phi_1$ connects $1$ and $p_1$ with speed $c_1$, and $\widehat{\phi}_1$ connects $1$ and $\widehat{p}_1$ with speed~$\widehat{c}_1$. Furthermore, by construction (see Theorems~\ref{th:upper_wave} and~\ref{exist}) and thanks to statement~$(ii)$ of Lemma~\ref{rewrite1}, we have that
$$\widehat{c}_1 \geq c_1, \qquad \widehat{p}_1 \geq p_1.$$
If $p_1 = \widehat{p}_1 $, then it is platform of $\mathcal{T}$. This happens in particular when $c_1 = \widehat{c}_1$, as one may check by applying twice the statement~$(i)$ of Lemma~\ref{rewrite1}. So consider the remaining case when $c_1 <  \widehat{c}_1 $ and $p_1 < \widehat{p}_1 $. Then the second traveling of $\widehat{\mathcal{T}}$ monotonically connects $\widehat{p}_1 \in (p_1, p_0)$ and $\widehat{p}_2$ with some speed $$\widehat{c}_2 \geq \widehat{c}_1 > c_1.$$
Applying Lemma~\ref{rewrite2}, one deduces that $\widehat{p}_2 \geq p_1$. Reiterating and since there is a finite number of steps, we end up proving that there exists some integer $j$ such that $$\widehat{p}_j = p_1.$$
In particular, $p_1$ is a platform of the propagating terrace $\widehat{\mathcal{T}}$. This concludes the proof.
\end{proof}
Let us now prove that actually the terrace $\widehat{\mathcal{T}}$ cannot have more platforms than $\mathcal{T}$:
\begin{prop} \label{uniqueness2}
The propagating terraces $\mathcal{T}$ and $\widehat{\mathcal{T}}$ share the same set of platforms, i.e. $J = \widehat{J}$ and $p_j = \widehat{p}_j$ for any $1 \leq j \leq J$.
\end{prop}
\begin{proof}
  	Let us prove that $p_1 = \widehat{p}_1$. According to Proposition~\ref{uniqueness1}, we know that there exists a positive integer $j$ such that $\widehat{p}_j = p_1$. Proceed by contradiction and assume that $j \geq 2$.

Again from Proposition~\ref{uniqueness1}, we also get that $\widehat{c}_1 > c_1$. Due to the ordering of the speeds of a propagating terrace, it follows that
\begin{equation}\label{eq:speedspeed}
\widehat{c}_j > c_1.
\end{equation}
Now, the trajectories $q_1$ and $\widehat{q}_j$ of the traveling waves $\phi_1$ and $\widehat{\phi}_j$ satisfy the following differential equations on $(p_1, \widehat{p}_{j-1})$:
	$$ {d q_1 \over dp} = -c_1 - {f(p) \over q_1},$$
	$$ {d\widehat{q}_j \over dp} = -\widehat{c}_j - {f(p) \over \widehat{q}_1},$$
together with
	$$q_1 ( p_1) = \widehat{q}_j (p_1) =  \widehat{q}_j (\widehat{p}_{j-1}) = 0 > q_1 (\widehat{p}_{j-1}).$$
	This latest inequality comes from the fact that $j \geq 2$ and $\widehat{p}_{j-1} \in (p_j,1) = (p_1,1)$. By substracting the two ODEs, we get
	$${d \over dp}(q_1 - \widehat{q}_j) = (\widehat{c}_j-c_i) - {f(p) \over q_1 \widehat{q}_j}(\widehat{q}_j - q_1).$$
	As a platform of the terrace $\mathcal{T}$ constructed in Section~\ref{sec:exist}, the steady state $p_1$ must be stable. Thus we can choose $\epsilon>0$ such that $f(p) < 0$ on $(p_1, p_1+\epsilon]$. We also choose $\delta \in (0, \epsilon)$, and define $p_\epsilon = p_1 +\epsilon$ and $p_\delta = p_1 + \delta$. Now, we integrate the above ODE from $p_\delta$ to $p_\epsilon$ and obtain that
\begin{equation}\label{eq:unique_almost}
(q_1(p_\epsilon)-\widehat{q}_j(p_\epsilon)) - (q_1(p_\delta)-\widehat{q}_j(p_\delta)) = \int_{p_\delta}^{p_\epsilon}(\widehat{c}_j-c_i) dp - \int_{p_\delta}^{p_\epsilon}{f(p) \over q_1 \widehat{q}_j}(\widehat{q}_j - q_1) dp.
\end{equation}
Recall that $\phi_1$ connects $p_0 =1$ and $p_1$, while $\widehat{\phi}_j$ connects $\widehat{p}_{j-1} \in (p_1, 1)$ and $\widehat{p}_j = p_1$. Therefore it follows from Lemma~\ref{rewrite2} that
\begin{equation}\label{eq:unique_almost_b}
q_1 < \widehat{q}_j  \ \ \mbox{ in the interval } \ \ (p_1, \widehat{q}_{j-1}).
\end{equation}
In particular, we have that
$$K_\epsilon := q_1(p_\epsilon)-\widehat{q}_j(p_\epsilon) < 0 ,$$
Since $p_\delta \to p_1$ as $\delta \to 0$, we also have $q_1(p_\delta)-\widehat{q}_j(p_\delta) \to 0$ as $\delta \to 0$. Thus, we can choose $\delta >0$ such that $q_1(p_\delta)-\widehat{q}_j(p_\delta) > K_\epsilon$. So, for such $\epsilon$ and $\delta$, we get that
$$ (q_1(p_\epsilon)-\widehat{q}_j(p_\epsilon)) - (q_1(p_\delta)-\widehat{q}_j(p_\delta)) < 0,$$
i.e. the left hand side of \eqref{eq:unique_almost} is negative.

On the other hand, for the right hand side, we have by \eqref{eq:speedspeed} that
 $\int_{p_\delta}^{p_\epsilon}(\widehat{c}_j-c_i) dp > 0$. Using again \eqref{eq:unique_almost_b}, we also have $\widehat{q}_j - q_1 \ge 0$. Since $q_1 \widehat{q}_j >0 $ and $f(p)<0$ on $[p_\delta, p_\epsilon]$, we get that $\int_{p_\delta}^{p_\epsilon}{f(p) \over q_1 \widehat{q}_j}(\widehat{q}_j - q_1) dp \le 0$, and the right hand side of \eqref{eq:unique_almost} is positive.

We have found a contradiction. We conclude that $j=1$ and $\widehat{p}_j = p_{1}$ and, by iteration, one eventually finds that $\mathcal{T}$ and $\widehat{\mathcal{T}}$ have the same platforms.
\end{proof}
We now know that both terraces have the same platforms. It immediatley follows from Lemma~\ref{speedunique} that, for each integer~$j$, the traveling waves $\phi_j$ and $\widehat{\phi}_j$ coincide up to some shift. Putting this uniqueness property together with the results of Section~\ref{sec:exist}, this ends the proof of Theorem~\ref{th:exist}.\\

\section*{}

\noindent{\bf Acknowledgements}

\noindent{This work was carried out in the framework of the CNRS International Research Network ``ReaDiNet''. The three authors were also supported by the joint PHC Star project MAP, funded by the French Ministry for Europe and Foreign Affairs and the National Research Fundation of Korea. The first author also acknowledges support from ANR via the project Indyana under grant agreement ANR- 21- CE40-0008.}

 \end{document}